\documentclass[12pt,a4paper]{amsart}
\usepackage{amsfonts,amsthm,amsmath,amscd,amssymb,bm, graphicx,mathrsfs}
\usepackage[all]{xy}
\usepackage[english]{babel}
\usepackage[usenames,dvipsnames,pdftex]{xcolor}
\usepackage{pdflscape}
\usepackage{multirow,multicol}
\usepackage{tikz,tkz-tab}
\usetikzlibrary{shapes,arrows,shadows}
\usetikzlibrary{calc}
\usepackage{ytableau}
\usepackage{hyperref}
\usepackage{diagbox}

\newtheorem{theorem}{Theorem}[section]
\newtheorem{theorem*}{Theorem}
\newtheorem{corollary}[theorem]{Corollary}
\newtheorem{corollary*}[theorem*]{Corollary}
\newtheorem{lemma}[theorem]{Lemma}
\newtheorem{proposition}[theorem]{Proposition}

\theoremstyle{definition}
\newtheorem{definition}[theorem]{Definition}
\newtheorem{remark}[theorem]{Remark}

\newtheorem*{question*}{Question}

\newtheorem*{conjecture*}{Conjecture}

\newtheorem*{notation*}{Notation}

\newtheorem*{claim*}{Claim}

\begin{document}
\title[Combinatorial formulas for Kazhdan-Lusztig polynomials]{Combinatorial formulas for Kazhdan-Lusztig polynomials with respect to $W$-graph ideals}
\author[Qi Wang]{Qi Wang}
\address{School of Mathematics\\Shanghai University of Finance and Economics\\777 Guoding Road\\Shanghai 200433\\China}
\email{q.wang@163.sufe.edu.cn}

\begin{abstract}
Let $(W,S,L)$ be a weighted Coxeter system and $J$ a subset of $S$, Yin \cite{Y15} introduced the weighted $W$-graph ideal $E_J$ and the weighted Kazhdan-Lusztig polynomials $ \left \{ P_{x,y} \mid x,y\in E_J\right \}$. In this paper, we study the combinatorial formulas for $P_{x,y}$, which will extend the results of Brenti \cite{B94} and Deodhar \cite{D97}.
\end{abstract}

\maketitle
\section{Introduction}
Let $(W,S)$ be a Coxeter system, we denote by $\ell$ and $\leqslant$ the length function and the Bruhat order on $W$, respectively.

$W$-graph is introduced by Kazhdan and Lusztig \cite{KL79}. A $W$-graph provides a method for constructing a matrix representation of $W$ and the entries of this matrix are called Kazhdan-Lusztig polynomials. However, these polynomials are not easy to compute.

Brenti \cite{B94} provided a combinatorial algorithm for Kazhdan-Lusztig polynomials, which is depended only on $R$-polynomials. Later, Deodhar \cite{D87} introduced the parabolic Kazhdan-Lusztig polynomials associated with a standard parabolic subgroup $W_J$ of $W$, and showed two analogous combinatorial formulas for these parabolic cases in \cite{D97}. Brenti's result corresponds exactly to the case  $J=\o $. Then, Tagawa \cite{T99} generalized these formulas to weighted parabolic Kazhdan-Lusztig polynomials.

Howlett and Nguyen \cite{HN12} introduced the concept of $W$-graph ideals and showed that a $W$-graph can be constructed from a $W$-graph ideal $E_J\subseteq  W$, where $W$ is equipped with the left weak Bruhat order $\leqslant _{L}$. They constructed the Kazhdan-Lusztig polynomials with respect to a $W$-graph ideal. Deodhar's construction exactly corresponds to the case $E_J=D_J$, where $D_{J}$ is the set of minimum coset representatives of $W_{J}$.

Recently, Yin (\cite{Y15} and \cite{Y16}) generalized the construction of $W$-graph ideals to a weighted Coxeter system $(W,S,L)$. He showed that there exists a pair of dual modules $M(E_J,L)$ and $\widetilde{M}(E_J,L) $ with respect to a given $W$-graph ideal $E_J$. Similarly, Yin constructed the weighted Kazhdan-Lusztig polynomials $\left \{ P_{x,y}\mid x,y\in E_J \right \}$ and the inverse weighted Kazhdan-Lusztig polynomials $\left \{ Q_{x,y}\mid x,y\in E_J \right \}$.

In this paper, we continue the work in \cite{Y15} and \cite{Y16} with slightly different on the ground ring. We consider Hecke algebras over the ring of Laurent polynomials $\mathbb{Z}[q_s^{1/2}, q_s^{-1/2}]$, where $q_{s}=q^{L(s)}$ and $L$ is a weight function.

This paper is organized as follows. In Section 2, we review basic concepts concerning weighted Coxeter groups and Hecke algebras. In Section 3 and 4, we modify the concept of $W$-graph ideals to our setting and the weighted Kazhdan-Lusztig polynomials on $M(E_J,L)$ are also considered. In Section 5 and 6, we devote to giving combinatorial formulas for $\left \{ P_{x,y}\mid x,y\in E_J \right \}$ and coefficients of those polynomials. In Section 7, we show similar combinatorial formulas for the inverse weighted Kazhdan-Lusztig polynomials.

\noindent\textit{Acknowledgements.} This project was considered during my master's course. I am very grateful to my supervisor, Prof. Yunchuan Yin, for taking me into this field and giving me a lot of help in writing this paper. I am also very grateful to him for his enthusiasm to help me study abroad.

\section{Preliminaries}

In this section, we follow the conventions in \cite{Y15}. Let $\Gamma$ be a totally ordered abelian group with zero element $\bm{0}$ and $\leqslant $ the order on $\Gamma$.

Let $\mathbb{Z}[\Gamma ]$ be a free $\mathbb{Z}$-module with a basis set $\left \{ q^{\gamma }\mid \gamma \in \Gamma  \right \}$ and the multiplication is given by $q^{\gamma }q^{\zeta }=q^{\gamma +\zeta }$,  where $q$ is an indeterminant. For any $f\in\mathbb{Z}[\Gamma ]$, we denote by $\left [ q^{\gamma } \right ]$ the coefficient of $f$ on $q^{\gamma}$ such that $f=\sum_{\gamma \in \Gamma }\left [ q^{\gamma } \right ]q^\gamma $. If $f\neq 0$, then the degree of $f$ is defined to be $\mathsf{deg}(f):=\text{max}\left \{ \gamma \mid \left [ q^{\gamma } \right ] \neq \bm{0}\right \}$, and we set $\mathsf{deg}(\bm{0})=-\infty $. Then, the map $\mathsf{deg}: \mathbb{Z}[\Gamma ]\rightarrow \Gamma \cup \left \{ -\infty  \right \}$ satisfies $\mathsf{deg}(fg)=\mathsf{deg}(f)+\mathsf{deg}(g)$.

For any reduced word $w=s_{1}s_{2}\cdots s_{k}$, a map $L: W\rightarrow \Gamma $ is called a weight of $W$ if
\begin{center}
$L(w)=L\left (s_{1} \right )+L\left (s_{2}  \right )+\cdots +L\left (s_{k}  \right )$.
\end{center}
Throughout, we assume that $L(s)\geqslant \bm{0}$ for any $s\in S$.

Let $\mathscr{H}:=\mathscr{H}\left ( W,S,L \right )$ be the weighted Hecke algebra corresponding to $(W,S)$ with parameters $\{ q_s^{1/2} \mid s\in S \}$, where $q_s=q^{L(s)}$. It is well-known that $\mathscr{H}$ is a free $\mathbb{Z}[\Gamma ]$-module and $\mathscr{H}$ has a basis set $\left \{ T_{w}\mid w\in W \right \}$. In particular, the multiplication is given by
\begin{center}
$T_{s}T_{w}=
\left\{\begin{aligned}
      &T_{sw}                            &\text{if}\ \ell(sw)>\ell(w)\\
      &q_{s}T_{sw}+(q_{s}-1)T_{w}  &\text{if}\ \ell(sw)<\ell(w).
\end{aligned}\right.$
\end{center}

We denote by $\bar{\ \ \ }: \mathbb{Z}[\Gamma ]\rightarrow\mathbb{Z}[\Gamma ]$ the automorphism of $\mathbb{Z}[\Gamma ]$ induced by sending $\gamma$ to $-\gamma$ for any $\gamma\in \Gamma$. This can be extended to a ring involution $\bar{\ \ \ }: \mathscr{H}\rightarrow \mathscr{H}$ such that
\begin{center}
$\sum\limits_{w\in W}^{ }\overline{a_{w}\cdot T_{w}}=\sum\limits_{w\in W}^{ }\overline{a_{w}}\cdot \overline{T_{w}}$,
\end{center}
and $\overline{T_{s}}=q_{s}^{-1}T_{s}+(q_{s}^{-1}-1)$ for any $s\in S$.

For any $J\subseteq S$, let $W_{J}:=\left \langle J \right \rangle$ be the corresponding parabolic subgroup of $W$. We denote the set of minimum coset representatives of $W_{J}$ by
\begin{center}
$D_{J}=\left \{ w\in W\mid \ell (ws)> \ell (w)\ \text{for all}\ s\in J\right \}$.
\end{center}
Let $\leqslant _L$ be the left weak Bruhat order on $W$. We say $x\leqslant _L y$ if and only if $y=zx$ and $\ell(y) =\ell(z) +\ell(x)$ for some $z\in W$. If this is the case, then $x$ is said to be a \emph{\textbf{suffix}} of $y$.

\begin{definition}(\cite[Definition 2.3]{Y15}) Let $E\subseteq W$, we define
\begin{center}
$\mathsf{Pos}\left (E\right ):=\left \{ s\in S\mid \ell(xs)> \ell (x)\ \text{for all}\ x\in E \right \} $.
\end{center}
\end{definition}

Obviously, $\mathsf{Pos}\left (E \right )$ is the largest subset $J$ of $S$ satisfying $E\subseteq D_{J}$. Let $E$ be an \emph{ideal} in poset $(W,\leqslant _L)$, that is, $E$ is a subset of $W$ such that if $u\in W$ is a suffix of $v\in E$, then $u=v$. This implies that $\mathsf{Pos}(E) =S/E=\left \{ s\in S|s\notin E \right \}$.

Let $J$ be a subset of $\mathsf{Pos}(E)$ such that $E\subseteq D_J$. In this context, we shall denote by $E_{J}$ the set $E$ with respect to $J$. For each $y\in E_{J}$, we define the following subsets of $S$.
\begin{center}
$\begin{aligned}
  SD\left ( y \right )&:=\left \{s\in S \mid \ell(sy)<\ell (y)\right \},\\
  SA\left ( y \right )&:=\left \{s\in S \mid \ell(sy)>\ell (y)\ \text{and}\ sy\in E_{J} \right \},\\
  WD\left ( y \right )&:=\left \{s\in S \mid \ell(sy)>\ell (y)\ \text{and}\ sy\notin D_{J} \right \},\\
  WA\left ( y \right )&:=\left \{s\in S \mid \ell(sy)>\ell (y)\ \text{and}\ sy\in D_{J}/E_{J} \right \}.
  \end{aligned}$
\end{center}
It is obvious that for each $y\in E_{J}$, each $s\in S$ appears in exactly one of the four sets defined above (since $E_J\subseteq D_{J}$).

\section{$W$-graph ideal modules}

We recall from \cite{Y16} the definition of $W$-graph ideals with some modifications.

\begin{definition}(\cite[definition 1.1]{Y16}, Modified) A $W$-graph consists of the following data.
\begin{enumerate}
  \item A vertex set $\Lambda $ together with a map $I$ which assigns a subset $I(x)\subseteq S$ to each vertex $x\in\Lambda $.
  \item For each $s\in S$ with $L(s)=\bm{0}$, there is an edge: $x\rightarrow sx$. For each $s\in S$ with $L(s)>\bm{0}$, there is a collection of edges such that
\begin{center}
$\overline{\mu _{x,y}^{s}}=\mu _{x,y}^{s}$ and $\left \{ \mu _{x,y}^{s}\in \mathbb{Z}[\Gamma  ]\mid  x,y\in \Lambda , s\in I(x), s\notin I(y) \right \}$.
\end{center}
  \item Let $[\Lambda ]_{\mathbb{Z}[\Gamma  ]}$ be a free $\mathbb{Z}[\Gamma  ]$-module with basis $\left \{ b_{y}\mid y\in \Lambda  \right \}$. We require that the map $T_{s}\rightarrow \tau _{s}$ gives an representation of $\mathscr{H}$ with the following multiplication.
  \begin{center}
    $\tau_{s}(b_{y})=
     \left\{\begin{aligned}
     &{b_{sy}}                                                                    & \text{if}\ L(s)&=\bm{0};\\
     &{-b_{y}}                                                                    & \text{if}\ L(s)&>\bm{0},s\in I(y);\\
     &{q_{s}b_{y}+q_{s}^{1/2}\sum_{x\in \Lambda ,s\in I(x)}^{ }\mu _{x,y}^{s}b_{x}} & \text{if}\ L(s)&>\bm{0},s\notin I(y).
     \end{aligned}\right.$
  \end{center}
\end{enumerate}
\end{definition}

\begin{definition}(\cite[Definition 2.4]{Y16}, Modified) Let $(W,S,L)$ be a weighted Coxeter system and $E_J\subseteq W$, we call $E_{J}$ a $W$-graph ideal if the following holds.
\begin{enumerate}
  \item There is a $\mathbb{Z}[\Gamma ]$-free $\mathscr{H}$-module $M(E_J,L)$ with a basis $\left \{ \Gamma _{y} \mid y\in E_{J}\right \}$ such that
  \begin{center}
   $T_{s}\Gamma _{y}=
      \left\{\begin{aligned}
      &q_{s}\Gamma _{sy}+\left ( q_{s}-1 \right )\Gamma _{y}            & \text{if}\ s&\in SD(y),\\
      &\Gamma _{sy}                                                     & \text{if}\ s&\in SA(y),\\
      &-\Gamma _{y}                                                     & \text{if}\ s&\in WD(y),\\
      &q_{s}\Gamma _{y}-\sum_{z<sy, z\in E_J}^{}r_{z,y}^{s}\Gamma _{z}  & \text{if}\ s&\in WA(y),
      \end{aligned}\right.$
  \end{center}
for some polynomials $r_{z,y}^{s}\in q_s\mathbb{Z}[\Gamma ]$.
  \item The $\mathscr{H}$-module $M(E_J,L)$ admits a $\mathbb{Z}[\Gamma ]$-semilinear involution $\Gamma _{y}\rightarrow \overline{\Gamma _{y}} $ satisfying $\overline{\Gamma _{1}}=\Gamma _{1} $ and $\overline{T_{w}\Gamma _{y}}=\overline{T_{w}}\ \overline{\Gamma _{y}}$ for all $T_{w}\in \mathscr{H}$, $\Gamma _{y}\in M(E_J,L)$.
\end{enumerate}
\end{definition}

We call $M(E_J,L)$ the $W$-graph ideal module and refer to \cite{Y16} for the definition of another $W$-graph ideal module $\widetilde{M}(E_J,L)$.

There exists an algebra homomorphism $\Phi :\mathscr{H}\rightarrow \mathscr{H}$ given by $\Phi \left ( q_{s} \right )=-q_{s}$ and $\Phi \left ( T_{w} \right )=\epsilon _{w}q_{w}\overline{T_{w}}$ for all $s\in S$, $w\in W$, where the bar is the standard involution on $\mathscr{H}$. Furthermore, $\Phi ^{2}=\text{id}$ and $\Phi$ commutes with the bar involution.

\begin{proposition}{\rm{(\cite[Theorem 3.1]{Y16})}} There exists a unique homomorphism $\eta $ from $M(E_J,L)$ to $\widetilde{M}(E_J,L)$ such that $\eta \left ( \Gamma _{1} \right )=\widetilde{\Gamma }_{1}$ and $\eta \left ( T_{w}\Gamma _{y} \right )=\Phi \left ( T_{w} \right )\eta \left (\Gamma _{y} \right )$ for all $T_{w}\in \mathscr{H}$ and $\Gamma _{y}\in M(E_J,L)$. Moreover,
\begin{enumerate}
  \item $\eta $ is a bijection and the inverse $\eta^{-1} $ satisfies properties of $\eta $.
  \item $\eta $ commutes with the involution on $M(E_J,L)$ and $\widetilde{M}(E_J,L)$.
\end{enumerate}
\end{proposition}
\begin{proof}
Let $\eta \left ( \Gamma_y \right )=\epsilon _{y}q_{y}\overline{\widetilde{\Gamma }_{y}}$, then it is easy to see the proposition.
\end{proof}

\section{Kazhdan-Lusztig polynomials with respect to $W$-graph ideals}

In this section, we recall the weighted Kazhdan-Lusztig polynomials introduced by Yin \cite{Y15} and \cite{Y16}. First, let $\Gamma ':=\left \{ L(s)\mid s\in S \right \}$ and we consider the totally order $\leqslant $ on $\Gamma$. Then, we have $\mathbb{Z}[\Gamma' ]\subseteq \mathbb{Z}[\Gamma ]$ and $\mathbb{Z}\left [ \Gamma ' \right ]=\left \{ \prod q^{n_i}_{s_i}\mid s_i\in S,n_i\in \mathbb{Z}\right \}$.

\subsection{The weighted $R$-polynomials on $M(E_J,L)$}
We recall from \cite{Y15} the construction of $\left \{ R_{x,y}\mid x,y\in E_J \right \}$ and $\{ \widetilde{R}_{x,y}\mid x,y\in E_J \} $ with some modifications.
\begin{definition}
There is a unique family of polynomials $R_{x,y}\in \mathbb{Z}[\Gamma ]$ satisfying
   \begin{center}
   $\overline{\Gamma _{y}}=\sum\limits_{x\in E_J}^{}\epsilon _{x}\epsilon _{y}q_{y}^{-1}R_{x,y}\Gamma _{x}$,
   \end{center}
we call them the weighted $R$-polynomials on $M(E_J,L)$.
\end{definition}

We assume that $R_{x,y}=0$ if $x\notin E_J$ or $y\notin E_J$. Then, we have

\begin{proposition}
Let $x,y\in E_J$, then $R_{x,y}=0$ if $x\nleqslant y$, $R_{x,y}=1$ if $x=y$.
\end{proposition}

\begin{proposition} For any $x,y\in E_J$, we have $R_{x,y}\in \mathbb{Z}[\Gamma' ]$ and
\begin{center}
$\bm{0}\leqslant\mathsf {deg}\left(R_{x,y}\right)\leqslant L(y)-L(x)$.
\end{center}
\end{proposition}

There are some further properties for $R_{x,y}$ as follows.

\begin{lemma}{\rm{(\cite[Lemma 4.3]{Y15}, \cite[Corollary 3.2]{Y16}, Modified)}} Let $x,y\in E_J$, then
\begin{enumerate}
  \item $\overline{R_{x,y}}=\epsilon _{x}\epsilon _{y}q_{x}q_{y}^{-1}\widetilde{R}_{x,y}$.
  \item $\sum\limits_{x\leqslant t\leqslant y, t\in E_J}^{}\epsilon _{t}\epsilon _{y}R_{x,t}
\widetilde{R}_{t,y}=\delta _{x,y}\ (\rm {Kronecker \ delta})$.
\end{enumerate}
 \end{lemma}
\begin{proof}
To prove the first statement, we apply the function $\eta$ to both sides of the formula
\begin{center}
$\overline{\Gamma _{y}}=\sum\limits_{x\in E_J}^{}\epsilon _{x}\epsilon _{y}q_{y}^{-1}R_{x,y}\Gamma _{x}$,
\end{center}
and use the fact that $\eta$ commutes with the involution to compare with the formula
\begin{center}
$\overline{\widetilde{\Gamma }_{y}}=\sum\limits_{x\in E_J}^{}\epsilon _{x}\epsilon _{y}q_{y}^{-1}\widetilde{R}_{x,y}\widetilde{\Gamma }_{x}$,
\end{center}
then we can see the result.

To prove the second statement, we use (1) and the equality $\overline{\overline{\Gamma}}_y=\Gamma_y$,
\begin{center}
$\begin{aligned}\Gamma_{y}
&=\sum_{x\in E_J}^{}\epsilon _{x}\epsilon _{y}q_y\overline{R_{x,y}\Gamma_{x}}=\sum_{t\leqslant y}^{}\sum_{x\leqslant t}^{}\epsilon _{x}\epsilon _{t}R_{x,t}\widetilde{R}_{t,y}\Gamma_{x}=\sum_{x\leqslant y}^{}\left(\sum_{x\leqslant t \leqslant y}^{}\epsilon _{x}\epsilon _{t}R_{x,t}\widetilde{R}_{t,y} \right)\Gamma_{x},
\end{aligned}$
\end{center}
Therefore, we have
\begin{center}
$\sum\limits_{x\leqslant t \leqslant y}^{}\epsilon _{t}\epsilon _{y}R_{x,t}\widetilde{R}_{t,y}=\epsilon_x^{-1}\epsilon_y\sum\limits_{x\leqslant t \leqslant y}^{}\epsilon _{x}\epsilon _{t}R_{x,t}\widetilde{R}_{t,y}=\epsilon_x^{-1}\epsilon_y\delta _{x,y}=\delta _{x,y}$.
\end{center}
\end{proof}

\begin{corollary}
Let $x,y\in E_J$, then
\begin{center}
$\sum\limits_{x<t\leqslant y, t\in E_J}^{}\epsilon _{x}\epsilon _{t}\widetilde{R}_{x,t}R_{t,y}
=\delta _{x,y}-R_{x,y}$.
\end{center}
\end{corollary}

We can easily find the following relation between $R_{x,y}$ and other kinds of $R$-polynomials.
\begin{remark}
With the above conventions, we have
\begin{enumerate}
  \item If $E_J=D_J$ and $x,y\in D_J$, then
\begin{center}
     $\begin{aligned}
       &R_{x,y}={{R}'}_{x,y}^{J}(q)_\psi ,         &u_s=-1,\\
       &\widetilde{R}_{x,y}={{R}'}_{x,y}^{J}(q)_\psi  ,   &u_s=q_s.
      \end{aligned}$
     \end{center}
where ${{R}'}_{x,y}^{J}(q)_\psi $ is a weighted parabolic $R$-polynomial (\cite{T99}), $u_s$ is a solution of the equation $u_{s}^{2}=q_{s}+\left ( q_{s}-1 \right )u_{s}$ and $\psi $ is a weight of $W$ into $\mathbb{Z}[\Gamma ]$.
  \item If $E_J=W$ (i.e. $J=\o$) and $x,y\in W$, then
\begin{center}
   $R_{x,y}=\widetilde{R}_{x,y}={{R}'}_{x,y} $,
\end{center}
where ${{R}'}_{x,y}$ is a weighted $R$-polynomial (\cite{L03}). If further $q_s=q$ for all $s\in S$, then
\begin{center}
   $R_{x,y}(q)=\widetilde{R}_{x,y}(q)=\mathbf{R}_{x,y}(q) $,
\end{center}
where $\mathbf{R}_{x,y}(q)$ is a $R$-polynomial (\cite{KL79}).
\end{enumerate}
\end{remark}

\subsection{The weighted Kazhdan-Lusztig polynomials on $M(E_J,L)$}

\begin{proposition} There exists a unique family of polynomials $\left \{ P_{x,y}\in \mathbb{Z}[\Gamma ]\mid x,y\in E_J \right \}$ satisfying the following condition:
\begin{center}
$q_{x}^{-1}q_{y}\overline{P_{x,y}}=\sum\limits_{x\leqslant t \leqslant y, t\in E_J}^{}R_{x,t}P_{t,y}$.
\end{center}
We call them the weighted Kazhdan-Lusztig polynomials on $M(E_J,L)$.
\end{proposition}
\begin{proof}
First, we show the existence of $P_{x,y}$ by induction on $\ell(y)-\ell(x)$. If $\ell(y)-\ell(x)=0$, then $\overline{P_{x,x}}=P_{x,x}\Leftrightarrow P_{x,x}=1$. If $\ell(y)-\ell(x)\neq0$, then it follows our assumption and Corollary 4.5 that
\begin{center}
$\begin{aligned}
\overline{\mathbb{F}}&:=\sum\limits_{x< t \leqslant y, t\in E_J}^{}\overline{R_{x,t}}\overline{P_{t,y}}\\
&=q_xq_{y}^{-1}\sum\limits_{x< t \leqslant y, t\in E_J}^{}\epsilon _{x}\epsilon _{t}\widetilde{R}_{x,t}\left (\sum\limits_{t\leqslant z \leqslant y, z\in E_J}^{}R_{t,z}P_{z,y}  \right )\\
&=q_xq_{y}^{-1}\sum\limits_{x< z \leqslant y, z\in E_J}^{}\left (\sum\limits_{x<t\leqslant z, t\in E_J}^{}\epsilon _{x}\epsilon _{t}\widetilde{R}_{x,t}R_{t,z} \right )P_{z,y} \\
&=q_xq_{y}^{-1}\sum\limits_{x< t \leqslant y, t\in E_J}^{}\delta _{x,t}P_{t,y}-q_xq_{y}^{-1}\sum\limits_{x< t \leqslant y, t\in E_J}^{}R_{x,t}P_{t,y}\\
&=-q_xq_{y}^{-1}\sum\limits_{x< t \leqslant y, t\in E_J}^{}R_{x,t}P_{t,y}\\
&=-q_xq_{y}^{-1}\mathbb{F}
\end{aligned}$
\end{center}
This implies
\begin{center}
$\overline{\mathbb{F}+P_{x,y}}=q_xq_{y}^{-1}\left ( \mathbb{F}+P_{x,y} \right )+\overline{\mathbb{F}}=q_xq_{y}^{-1}P_{x,y}$,
\end{center}

Therefore, the existence of $P_{x,y}$ is shown by replacing $\mathbb{F}$. The uniqueness of $P_{x,y}$ is obvious by the method of the proof of the existence above.
\end{proof}

Similarly, we set $P_{x,y}=0$ if $x\notin E_J$ or $y\notin E_J$.
\begin{corollary}
Assume that $x<y\in E_J$, then $P_{x,y}\in \mathbb{Z}[\Gamma' ]$ and
\begin{center}
$\bm{0}\leqslant \mathsf{deg}  \left(P_{x,y}\right)<\frac{L(y)-L(x)}{2}$.
\end{center}
\end{corollary}
\begin{proof}
It follows Proposition 4.3 and Proposition 4.7 that $P_{x,y}\in \mathbb{Z}[\Gamma' ]$ and the inequality on the left is true. Let $d:= \mathsf{deg} \left(P_{x,y}\right)$ and note that
\begin{center}
$q_{x}^{-1/2}q_{y}^{1/2}\overline{P_{x,y}}-q_{x}^{1/2}q_{y}^{-1/2}P_{x,y}=q_{x}^{1/2}q_{y}^{-1/2}\sum\limits_{x<t \leqslant y, t\in E_J}^{}R_{x,t}P_{t,y}$,
\end{center}
we have
\begin{center}
$\frac{L(y)-L(x)}{2}-d\leqslant \text{deg}  \left(q_{x}^{-1/2}q_{y}^{1/2}\overline{P_{x,y}}\right)\leqslant \frac{L(y)-L(x)}{2}$,

$\frac{L(x)-L(y)}{2}\leqslant \text{deg}  \left(q_{x}^{1/2}q_{y}^{-1/2}P_{x,y}\right)\leqslant \frac{L(x)-L(y)}{2}+d$.
\end{center}
Then, by the uniqueness of $P_{x,y}$,
\begin{center}
$\frac{L(x)-L(y)}{2}+d<\textbf{0}<\frac{L(y)-L(x)}{2}-d$.
\end{center}
Therefore, the inequality on the right is also true.
\end{proof}

\begin{remark}
It is worth mentioning that the $\mathscr{H}$-module $M(E_J,L)$ has a unique basis $\left \{ C_{y}\mid y\in E_J \right \}$ such that $\overline {C_{y}}=C_{y}$ for any $y\in E_J $, and
\begin{center}
  $C_{y}=\sum\limits_{x\leqslant y\in E_J}^{}\epsilon _{x}\epsilon _{y}q_{x}^{-1}q_{y}^{1/2}\overline {P_{x,y}}\Gamma _{x}$.
\end{center}
This is the so-called Kazhdan-Lusztig basis.
 \end{remark}

The relation between $P_{x,y}$ and other kinds of Kazhdan-Lusztig polynomials is very similar to Remark 4.6, we do not show details.

\section{The combinatorial formulas for $P_{x,y}$}

In this section, we define $\mathscr{R}$-polynomials on $M(E_J,L)$ and show combinatorial formulas for $\left \{ P_{x,y} \mid x,y\in E_J \right \}$, which extend the results of Tagawa \cite{T99} and Deodhar \cite{D97}.

For any $\gamma \in \Gamma $, we define the following truncation functions.
\begin{center}
  $U_{\zeta }\left ( \sum\limits_{\gamma \geqslant \bm{0}}^{} \left [ q^{\gamma } \right ]q^{\gamma }\right )=\sum\limits_{\gamma >\zeta }^{} \left [ q^{\gamma } \right ]q^{\gamma }$ and $L_{\zeta }\left ( \sum\limits_{\gamma \geqslant \bm{0}}^{} \left [ q^{\gamma } \right ]q^{\gamma }\right )=\sum\limits_{\gamma <\zeta }^{} \left [ q^{\gamma } \right ]q^{\gamma }$.
\end{center}

\begin{definition}
Assume that $E_{J}$ is a $W$-graph ideal.
\begin{enumerate}
  \item Let $\mathscr{J}_{k}(x,y)=\left \{ \varphi :x=x_0< x_{1}<\cdots <x_{k+1}=y\mid x_i\in E_{J}\right \}$ be the set of all $E_J$-chains of length equal to $k+1$, where $x_0$ is called the initial element of $\varphi $ and $x_{k+1}$ is called the final element of $\varphi $. Let $x<y$, we denote by $\mathscr{J}(x,y)=\bigcup_{k\geqslant 0}^{}\mathscr{J}_{k}(x,y)$ the set of all $E_J$-chains.
  \item Let $\mathscr{M}_{k}(x,y)=\left \{ \varphi :x=x_0\leqslant  x_{1}\leqslant \cdots \leqslant x_{k+1}=y\mid x_i\in E_{J} \right \}$ be the set of all $E_J$-multichains of length equal to $k+1$, and we denote by $\mathscr{M}(x,y)=\bigcup_{k\geqslant 0}^{}\mathscr{M}_{k}(x,y)$ the set of all $E_J$-multichains.
\end{enumerate}
\end{definition}

Note that an $E_J$-chain is an $E_J$-multichain as well. Conversely, for each $E_J$-multichain, there exists a unique $E_J$-chain which is obtained by removing the repetitions.

We define a family of polynomials which is depended only on the weighted $R$-polynomials $R_{x,y}$. We call them $\mathscr{R}$-polynomials on $M(E_J,L)$.

\begin{definition}
Assume that $x\leqslant y\in E_{J}$ and $\varphi :x=x_0\leqslant  x_{1}\leqslant \cdots \leqslant x_{r+1}=y$ is a $E_J$-multichain, let
\begin{center}
  $\mathscr{R}_\varphi=\left\{\begin{aligned}
   &q_{x}^{-1}q_{y}\overline {R_{x,y}}                                      &\text{if}\ r=0 ,\\
   &\mathscr{R}_{x,x_1}U_{\frac{L(y)-L(x_1)}{2}} \left( q_{x_1}^{-1}q_y
    \overline {\mathscr{R}_{\varphi^{'}}}\right)                   &\text{if}\ r \geqslant1.
   \end{aligned}\right.$
\end{center}
where $\varphi':x_{1}\leqslant  x_{2}\leqslant \cdots \leqslant y \in \mathscr{M}(x_1,y)$.
\end{definition}

Then,  we show some properties of $\mathscr{R}$-polynomials.
\begin{proposition}
Assume that $x\leqslant y\in E_{J}$, $\varphi\in \mathscr{M}(x,y)$, $\varphi' \in \mathscr{M}(x_1,y)$ and $\mathscr{R}_\varphi\neq 0$.
\begin{enumerate}
  \item Let $d(\varphi )$ be the maximum power of $q$ that divides $\mathscr{R}_\varphi$, then
  \begin{center}
   $\frac{L(y)-L(x_1)}{2} <  \mathsf{deg}\left(\mathscr{R}_\varphi\right)\leqslant L(y)-L(x)-d(\varphi' )$.
  \end{center}
  \item For all $i\in \left \{ 2, 3, \cdots ,\ell(\varphi) \right \}$, we have $x_{i-1} < x_i$.

\end{enumerate}
\end{proposition}
\begin{proof}
(1) The inequality on the left follows the definition of $\mathscr{R}_\varphi$. Since
\begin{center}
$\begin{aligned}deg(\mathscr{R}_\varphi)
      &=\mathsf{deg}( \mathscr{R}_{x,x_1})+\mathsf{deg}\left( U_{\frac{L(y)-L(x_1)}{2}} \left( q_{x_1}^{-1}q_y\overline {\mathscr{R}_{\varphi^{'}}}\right)\right)\\
      &\leqslant L(x_1)-L(x)+\mathsf{deg}\left( q_{x_1}^{-1}q_y\overline {\mathscr{R}_{\varphi^{'}}}\right)=L(y)-L(x)-d(\varphi' ),
\end{aligned}$
\end{center}
the inequality on the right is also true.

(2) We show this by induction on $\ell(\varphi)$. If $\ell(\varphi)=2$, then $\varphi =\left ( x_0,x_1,x_{2} \right )$ and $x_1<x_2$ (there is a contradiction between $\mathscr{R}_{x_0,x_1,x_2}=0$ and our assumption if $x_1=x_2$).

Suppose that the inequality holds when $\ell(\varphi)\leqslant k, (k\geqslant 2)$, then we show it for $\ell(\varphi)=k+1$. Since $\mathscr{R}_{x_0,x_1,\cdots , x_{k+1}}\neq0$,
\begin{center}
$\mathsf{deg}\left(q_{x_1}^{-1}q_{x_{k+1}}\mathscr{R}_{x_2,\cdots, x_{k+1}}\right)>\frac{L(y)-L(x_1)}{2}$.
\end{center}
Then, by statement (1), we have $L(x_1)<L(x_2)$. This is equivalent to $x_1<x_2$. Therefore,
\begin{center}
$x_1<x_2<x_3<\cdots<x_{k+1}$.
\end{center}
follows our inductive hypothesis.
\end{proof}

The following corollaries are obvious.

\begin{corollary}
Let $x\leqslant y\in E_{J}$ and $\varphi \in \mathscr{M}(x,y)$. If $\ell(y)-\ell(x)<\ell(\varphi)-1$, then $\mathscr{R}_\varphi= 0$.
\end{corollary}

\begin{corollary}
For any $x\leqslant z \in E_{J}$ and $k\in \mathbb{N}$, let $z^{(k)}:=(z,z,...,z)\in ({E_J})^k$. Then, $\mathscr{R}_{z^{(k)}}=0$ if $k\geqslant 3$, $\mathscr{R}_{x,z^{(k)}}=0$ if $k\geqslant 2$.
\end{corollary}

\begin{corollary}
For any $x, y \in E_{J}$, we have
\begin{center}
$\overline{\mathscr{R}_{x,y}}=\epsilon _{x}\epsilon _{y}q_{x}q_{y}^{-1}\widetilde{\mathscr{R}}_{x,y}$,
\end{center}
where $\widetilde{\mathscr{R}}_{x,y}$ is the $\mathscr{R}$-polynomials on $\widetilde {M}(E_J,L)$.
\end{corollary}

In order to prove our main result in this section, we require the following lemmas.
\begin{lemma} Let $x<y\in E_J$, then
\begin{center}
     $P_{x,y}^{}-\mathscr{R}_{x,y}=\sum\limits_{x\leqslant t< y, t\in E_J}^{}q_{t}^{-1}q_y\mathscr{R}_{x,t}\overline{P_{t,y}}$.
\end{center}
\end{lemma}
\begin{proof}
By applying the involution $\bar {\ }$ on $q_{x}^{-1}q_{y}\overline{P_{x,y}}=\sum\limits_{x\leqslant t \leqslant y, t\in E_J}^{}R_{x,t}P_{t,y}$, one can get the required result.
\end{proof}

\begin{lemma}
For any $x<y\in E_J$, we have
\begin{center}
$\left [ q^{\bm{0}} \right ]\left ( \sum\limits_{\varphi\in \mathscr{M}\left ( x,y \right )}^{}  \mathscr{R}_\varphi \right )=\left\{\begin{aligned}
   &\left [ q_{x}^{-1}q_{y} \right ]R_{x,y}    &\text{if}\ \ell(\varphi)&=1\\
   &0                                             &\text{if}\ \ell(\varphi)&\geqslant 2,
   \end{aligned}\right.$
\end{center}
\end{lemma}
\begin{proof}
It is easy to prove by the definition of $\mathscr{R}_\varphi$.
\end{proof}

We now have the following result as described in the introduction.
\begin{theorem} Assume that $E_J$ is a $W$-graph ideal and $x,y\in E_J$.
\begin{enumerate}
  \item If $x<y$, then $P_{x,y}^{}=L_{\frac{L(y)-L(x)}{2} }\left ( \sum\limits_{\varphi\in \mathscr{J}\left ( x,y \right )}^{}  \mathscr{R}_\varphi\right )$.
  \item If $x \leqslant y$, then $P_{x,y}^{}=\sum\limits_{\varphi\in \mathscr{M}\left ( x,y \right )}^{} \mathscr{R}_\varphi$.
\end{enumerate}
\end{theorem}

\begin{proof}
According to Corollary 5.4, the sum on the right side of (1) and (2) is finite. We show $\left ( 1 \right )$ by induction on $\ell(y)-\ell(x)$. If  $\ell(y)-\ell(x)=1$, then we use the involution on two sides of the equation introduced in Proposition 4.7, and we get $P_{x,y}-q_{x}^{-1}q_y\overline {P_{x,y}}=q_{x}^{-1}q_y\overline{R_{x,y}}$. It follows the definition of $\mathscr{R}_\varphi$ and Proposition 4.8 that $P_{x,y}^{}=L_{\frac{L(y)-L(x)}{2} }\left ( \mathscr{R}_{x,y}\right )$.

Assume that $\left ( 1 \right )$ holds when $\ell(y)-\ell(x)< k$ and we show it for $\ell(y)-\ell(x)= k$. First, by Lemma 5.7, we have
\begin{center}
$\left [ q^{\bm{0}} \right ]\left (P_{x,y}^{}\right )-\left [ q^{\bm{0}}\right ]\left (\mathscr{R}_{x,y}\right )=\left [ q^{\bm{0}} \right ]\left (\sum\limits_{x\leqslant t< y,t\in E_J}^{}q_{t}^{-1}q_y\mathscr{R}_{x,t}\overline{P_{t,y}}\right )$.
\end{center}
Since $\mathsf{deg}\left(P_{x,y}\right)< \frac{L(y)-L(x)}{2} $ and Lemma 5.8, it is easy to check that
\begin{center}
$\left [q^{\bm{0}}\right ]\left(P_{x,y}^{}\right)=\left [ q^{\bm{0}} \right ]\left(L_{\frac{L(y)-L(x)}{2} }\left ( \sum\limits_{\varphi\in \mathscr{J}\left ( x,y \right )}^{}  \mathscr{R}_\varphi\right )\right)$.
\end{center}

Assume again that the following holds for all pairs ${x}'<{y}'$ and ${\gamma }'<\gamma \ (q^{\gamma} \in\mathbb{Z}[\Gamma ' ])$,
\begin{center}
   $\left [q^{\gamma'}  \right ]\left(P_{x,y}^{}\right )-\left [q^{\gamma'} \right ]\left(q_{x}^{-1}q_y\overline {P_{x,y}}\right )=\left [q^{\gamma'}  \right ]\left (\sum\limits_{\varphi\in \mathscr{J}\left ( x,y \right )}^{} \mathscr{R}_\varphi\right )$.
\end{center}
Similarly, by Lemma 5.7, we have
\begin{center}
  $\begin{aligned}\left [ q^\gamma  \right ]\left(P_{x,y}^{}\right)-&\left [ q^\gamma  \right ]\left(\mathscr{R}_{x,y}\right)-\left [ q_x^{-1}q_yq^{-\gamma } \right ]\left(P_{x,y}^{}\right)\\
  &=\sum\limits_{x<t<y,t\in E_J}^{}\sum\limits_{0\leqslant \xi \leqslant \gamma }^{}\left [ q^{\gamma -\xi } \right ]\left(\mathscr{R}_{x,t}\right)\left [ q^\xi  \right ]\left(q_t^{-1}q_y\overline {P_{t,y}^{}}\right),
   \end{aligned}$
\end{center}
and
\begin{center}
$\left [ q^\xi  \right ]\left(q_t^{-1}q_y\overline {P_{t,y}^{}}\right)=\left\{\begin{aligned}
   &0                                                               &\text{if}\ \xi &\leqslant \frac{L(y)-L(t)}{2},\\
   &\left [ q_t^{-1}q_yq^{-\xi } \right ]\left(P_{t,y}^{}\right)    &\text{if}\ \xi &>\frac{L(y)-L(t)}{2}.
\end{aligned} \right. $
\end{center}
Then, we substitute the latter one into the former one,
\begin{center}
 $ \begin{aligned}\left [ q^\gamma \right ]\left(P_{x,y}^{}\right)-&\left [ q^\gamma \right ]\left(\mathscr{R}_{x,y}\right)-\left [ q_x^{-1}q_yq^{-\gamma} \right ]\left(P_{x,y}^{}\right)\\
&=\sum\limits_{x<t<y,t\in E_J}^{}\sum\limits_{\frac{L(y)-L(t)}{2}<\xi\leqslant \gamma}^{}\left [ q^{\gamma-\xi} \right ]\left(\mathscr{R}_{x,t}\right)\left [ q_t^{-1}q_yq^{-\xi} \right ]\left(P_{t,y}^{}\right).
 \end{aligned}$
 \end{center}

For any $x<t<y$ and $\frac{L(y)-L(t)}{2} <\xi \leqslant \gamma $, we have $L(y)-L(t)-\xi<\xi<\gamma$. Then, following our assumption,
\begin{center}
   $\left [ q_t^{-1}q_yq^{-\xi} \right ]\left(P_{t,y}^{}\right)-\left [q^\xi \right ]\left(P_{t,y}^{}\right)=\left [ q_t^{-1}q_yq^{-\xi} \right ]\left(\sum\limits_{\varphi^{'}\in \mathscr{J}\left ( t,y \right )}^{} \mathscr{R}_{\varphi^{'}}\right)$.
\end{center}
However, $\left [q^\xi \right ]\left(P_{x,y}^{}\right)=0$ (since $\xi\geqslant \frac{L(y)-L(t)}{2} $). By the definition of $\mathscr{R}_{\varphi}$, Proposition 5.3 and Corollary 5.4, we have
\begin{center}
$\begin{aligned}
    \left [ q^\gamma \right ]&\left(P_{x,y}^{}\right)-\left [ q^\gamma \right ]\left(\mathscr{R}_{x,y}\right)-\left [ q_x^{-1}q_yq^{-\gamma} \right ]\left(P_{x,y}^{}\right)\\
    &=\sum\limits_{x<t<y,t\in E_J}^{}\sum\limits_{\varphi^{'}\in \mathscr{J}\left ( t,y \right )}^{}\sum\limits_{\frac{L(y)-L(t)}{2} <\xi \leqslant \gamma}^{}\left [ q^{\gamma-\xi } \right ]\left(\mathscr{R}_{x,t}\right)\left [ q_t^{-1}q_yq^{-\xi } \right ]\left(\mathscr{R}_{\varphi^{'}}\right)\\
    &=\sum\limits_{x<t<y,t\in E_J}^{}\sum\limits_{\varphi^{'}\in \mathscr{J}\left ( t,y \right )}^{}\left [ q^\gamma \right ]\left(\mathscr{R}_{x,\varphi^{'}}\right)=\sum\limits_{\varphi\in \mathscr{J}\left ( x,y \right )}^{}\left [ q^\gamma \right ]\left(\mathscr{R}_\varphi\right)-\sum\limits_{\varphi\in \mathscr{J}_1\left ( x,y \right )}^{}\left [ q^\gamma \right ]\left(\mathscr{R}_\varphi\right)\\
    &=\left [ q^\gamma \right ]\left(\sum\limits_{\varphi\in \mathscr{J}\left ( x,y \right )}^{}\mathscr{R}_\varphi\right)-\left [ q^\gamma \right ]\left(\sum\limits_{\varphi\in \mathscr{J}_1\left ( x,y \right )}^{}\mathscr{R}_\varphi\right)=\left [ q^\gamma \right ]\left(\sum\limits_{\varphi\in \mathscr{J}\left ( x,y \right )}^{}\mathscr{R}_\varphi\right)-\left [ q^\gamma \right ]\left(\mathscr{R}_{x,y}\right)
\end{aligned}$
\end{center}

The above is equivalent to $P_{x,y}^{}-q_{x}^{-1}q_y\overline {P_{x,y}}=\sum\limits_{\varphi\in \mathscr{J}\left ( x,y \right )}^{} \mathscr{R}_\varphi$. Therefore, (1) is easy to prove by Corollary 4.8.

Now, we can show the statement (2). If $x=y$, then (2) is true since $P_{x,x}^{}=1=\mathscr{R}_{x,x}$. If $x<y$, then by $\left ( 1 \right )$, Proposition 5.3, Corollary 5.4, Corollary 5.5 and the fact $U_{\alpha+\beta } \left (q^\alpha P\right )=q^\alpha \overline{L_{-\beta }(\overline {P})}$, we have
   \begin{center}
    $\begin{aligned}P_{x,y}^{}
    &=\sum\limits_{\varphi\in \mathscr{J}\left ( x,y \right )}^{}\mathscr{R}_\varphi+q_x^{-1}q_y\overline {L_{\frac{L(y)-L(x)}{2}}(\sum\limits_{\varphi\in \mathscr{J}\left ( x,y \right )}^{}\mathscr{R}_\varphi)}\\
    &=\sum\limits_{\varphi\in \mathscr{J}\left ( x,y \right )}^{}\mathscr{R}_\varphi+\sum\limits_{\varphi\in \mathscr{J}\left ( x,y \right )}^{}U_{\frac{L(y)-L(x)}{2}}(q_x^{-1}q_y\overline {\mathscr{R}_\varphi})\\
    &=\sum\limits_{\varphi\in \mathscr{J}\left ( x,y \right )}^{}\mathscr{R}_\varphi+\sum\limits_{\varphi\in \mathscr{J}\left ( x,y \right )}^{}\mathscr{R}_{x,\varphi}=\sum\limits_{x<t< y,\varphi\in \mathscr{J}\left ( t,y \right ),t\in E_J}^{}\mathscr{R}_{x,\varphi}+\sum\limits_{\varphi\in \mathscr{J}\left ( x,y \right )}^{}\mathscr{R}_{x,\varphi}+\sum\limits_{\varphi\in \mathscr{J}_1\left ( x,y \right )}^{}\mathscr{R}_\varphi\\
    &=\sum\limits_{x\leqslant t<y,\varphi\in \mathscr{J}\left ( x,y \right ),t\in E_J}^{}\mathscr{R}_{x,\varphi}+\sum\limits_{\varphi\in \mathscr{J}_1\left ( x,y \right )}^{}\mathscr{R}_\varphi=\sum\limits_{x\leqslant t\leqslant y,\varphi\in \mathscr{M}\left ( x,y \right ),t\in E_J}^{}\mathscr{R}_{x,\varphi}+\sum\limits_{\varphi\in \mathscr{M}_1\left ( x,y \right )}^{}\mathscr{R}_\varphi\\
    &=\sum\limits_{\varphi\in \mathscr{M}\left ( x,y \right )}^{}\mathscr{R}_\varphi
    \end{aligned}$
   \end{center}
   This completes the proof of $\left ( 2 \right )$.
\end{proof}

Moreover, we have the following corollary immediately. Let $\Psi =\bigcup\limits_{k\leqslant \ell(y)-\ell(x)}\mathscr{J}_k \left ( x,y \right )$ and $\Upsilon =\bigcup\limits_{k\leqslant \ell(y)-\ell(x)}\mathscr{M}_k\left ( x,y \right )$.
\begin{corollary}
Let $x,y\in E_J$.
\begin{enumerate}
  \item If $x < y$, then $P_{x,y}^{}=L_{\frac{L(y)-L(x)}{2}}\left ( \sum\limits_{\varphi\in\Psi }^{}  \mathscr{R}_\varphi\right )$.
  \item If $x \leqslant y$, then $P_{x,y}^{}=\sum\limits_{\varphi\in \Upsilon }^{} \mathscr{R}_\varphi$.
\end{enumerate}
\end{corollary}
\begin{proof}
It follows Proposition 5.4 that $\mathscr{R}_\varphi= 0$ if $\varphi\in \bigcup\limits_{k\geqslant \ell(y)-\ell(x)+1}\mathscr{M}_k \left ( x,y \right )  $. Then, the corollary follows Theorem 5.9.
\end{proof}

\section{The coefficients of $P_{x,y}$}
The purpose in this section is to obtain explicit formulas for the coefficients of $P_{x,y}$. Let
\begin{center}
$\Gamma'':=\left \{ \sum n_iL(s_i) \mid s_i\in S,n_i\in \mathbb{Z},i\in \mathbb{N}\right \}$.
\end{center}
Then,  by Theorem 5.9 and Corollary 5.10, we immediately have the following results.

\begin{corollary}
Assume that $x,y\in E_J$ and $\gamma \in\Gamma''$.
\begin{enumerate}
  \item If $x < y$, then $\left [ q^\gamma  \right ]\left (P_{x,y}^{}\right )=\left [ q^\gamma \right ]\left (L_{\frac{L(y)-L(x)}{2}}\left ( \sum\limits_{\varphi\in \Psi}^{}  \mathscr{R}_\varphi\right )\right )$.
  \item If $x \leqslant y$, then $\left [ q^\gamma \right ]\left (P_{x,y}^{}\right )=\left [ q^\gamma \right ]\left (\sum\limits_{\varphi\in \Upsilon}^{} \mathscr{R}_\varphi\right )$.
\end{enumerate}
\end{corollary}

Before to show the formulas, we have to show the following.
\begin{lemma}
Assume that $x\leqslant y\in E_J$, $r\in \mathbb{N}$, $\varphi \in \mathscr{M}\left ( x,y \right )$ and $\gamma \in\Gamma''$, then
\begin{center}
$\left [ q^\gamma  \right ]\left(\mathscr{R}_\varphi\right )=\sum\limits_{S\in \mathscr{F}_{\gamma }(\varphi )}^{}\prod\limits_{i=0}^{r}\left [ q_{x_{i+1}}q_{y}^{-1}q^{\lambda _i}q^{\lambda _{i+1}} \right ]\left(R_{x_i,x_{i+1}}\right )$.
   \end{center}
where we set $S:=(\lambda _0,\lambda _1,\cdots,\lambda _{r+1})$ and
\begin{center}
    $ \begin{aligned}
\mathscr{F}_{\gamma }(\varphi ) := \{ &(a_0,a_1,\cdots,a_{r+1})\in {(\Gamma'')}^{r+2}\mid \\
              &a_0=L(y)-L(x)-\gamma ,\\
              &\gamma \geqslant a_1> a_2> \cdots>a_r>a_{r+1}=\bm{0}, \\
              &a_i>L(y)-L(x_i)-a_i\geqslant a_{i+1}\ for\  i\in\left \{ 1,2,\cdots ,r \right \} \}.
    \end{aligned}$
\end{center}
\end{lemma}

\begin{proof}
It follows the definition of $\mathscr{R}_\varphi$ and $\lambda _0=L(y)-L(x)-\gamma $ that
\begin{center}
$\begin{aligned}  \left [ q^\gamma  \right ]&\left(\mathscr{R}_\varphi\right )
         =\sum\limits_{\gamma \geqslant \lambda_1\geqslant \bm{0}}\left [ q^{\gamma -\lambda_1} \right ]\left(\mathscr{R}_{x,x_{1}}\right )\left [ q_{x_1}^{-1}q_yq^{-\lambda_1} \right ]\left( \mathscr{R}_{\varphi^{'}}\right)\\
         &=\sum\limits_{\substack{ \gamma \geqslant \lambda_1\geqslant \bm{0},\\ S'\in \mathscr{F}_{L(y)-L(x_1)-\lambda_1}(\varphi ')}}\left [ q_{x}^{-1}q_{x_1}q^{\lambda_1}q^{-\gamma } \right ]\left(R_{x,x_{1}}\right )\left( \prod\limits_{i=1}^{r}\left [ q_{x_{i+1}}q_{y}^{-1}q^{\lambda_i}q^{\lambda_{i+1}} \right ]\left(R_{x_i,x_{i+1}}\right )\right)\\
         &=\sum\limits_{S\in \mathscr{F}_{\gamma }(\varphi )}^{}\prod\limits_{i=0}^{r}\left [ q_{x_{i+1}}q_{y}^{-1}q^{\lambda_i}q^{\lambda_{i+1}} \right ]\left(R_{x_i,x_{i+1}}\right ),
 \end{aligned}$
\end{center}
where $S'=(\lambda_1,\lambda_2,\cdots,\lambda_{r+1})$ and
\begin{center}
    $ \begin{aligned}
\mathscr{F}_{L(y)-L(x_1)-\lambda_1}(\varphi' ) = \{ &(a_1,a_2,\cdots,a_{r+1})\in {(\Gamma'')}^{r+1}\mid \\
              &a_1=\lambda_1,\\
              &L(y)-L(x_1)-\lambda_1\geqslant a_2> a_3> \cdots>a_r>a_{r+1}=\bm{0}, \\
              &a_i> L(y)-L(x_i)-a_i\geqslant a_{i+1}\ for\  i\in\left \{ 2,3,\cdots ,r \right \} \}.
    \end{aligned}$
\end{center}
This completes the proof of the lemma.
\end{proof}

\begin{theorem}
Assume that $x,y\in E_J$, $\varphi \in \mathscr{M}(x,y)$, $r\in N$ and $\gamma \in\Gamma''$, then
\begin{center}
$\begin{aligned}
\left [ q^\gamma  \right ]\left (P_{x,y}^{}\right )
=&\left [ q_x^{-1}q_yq^{-\gamma } \right ] \left ( R_{x,y}\right )\\
&+\sum\limits_{r=1}^{\ell(y)-\ell(x)}\sum\limits_{\varphi\in \mathscr{M}\left ( x,y \right )}^{}\sum\limits_{S\in \mathscr{F}_{\gamma }(\varphi )}^{}\prod\limits_{i=1}^{r}\left [ q_{x_{i+1}}q_{y}^{-1}q^{\lambda_i}q^{\lambda_{i+1}} \right ]\left(R_{x_i,x_{i+1}}\right )
\end{aligned}$.
\end{center}
\end{theorem}
\begin{proof}
Following Corollary 5.4 and Theorem 5.9, one can easily check that
\begin{center}
$\left [ q^\gamma  \right ]\left (P_{x,y}^{}\right )
=\left [ q^\gamma  \right ] \left ( \mathscr{R}_{x,y}\right ) +\sum\limits_{r=1}^{\ell(y)-\ell(x)}\sum\limits_{\varphi\in \mathscr{M}\left ( x,y \right )}^{}\left [ q^\gamma  \right ]\left ( \mathscr{R}_\varphi\right )$.
\end{center}
The result is a straightforward consequence of Lemma 6.2 and the definition of $\mathscr{R}_\varphi$.
\end{proof}

\section{The inverse weighted Kazhdan-Lusztig polynomials}
In this section, we recall from \cite{Y16} the construction of $\left \{ Q_{x,y}\mid x,y\in E_J \right \}$ and give combinatorial formulas for those polynomials, which are similar to $\left \{ P_{x,y} \mid x,y\in E_J\right \}$. This also extends the results of \cite{T01} and \cite{D97}.

Let $y\in E_J$, the formula for $C_{y}$ introduced in Remark 4.9 may be rewritten as
\begin{center}
  $q_{y}^{1/2}C_{y}=\sum\limits_{x\leqslant y, x\in E_J}^{}\epsilon_{x}\epsilon_{y}P_{x,y}q_{x}\overline {\Gamma _{x}}$,
\end{center}
and inverting this gives
\begin{center}
  $q_{y}\overline {\Gamma _{y}}=\sum\limits_{x\leqslant y,x\in E_J}^{}Q_{x,y}q_{x}^{1/2}C_{x}$,
\end{center}
where $Q_{x,y}$ is given recursively by
\begin{center}
  $  \sum\limits_{x\leqslant t\leqslant y, t\in E_J}^{}\epsilon _{t}\epsilon _{y}Q_{x,t}P_{t,y}=\delta _{x,y}$.
\end{center}

\begin{proposition}
There exists a unique family of polynomials $\left \{ Q_{x,y}\in \mathbb{Z}[\Gamma' ]\mid x,y\in E_J \right \}$ satisfying $Q_{x,y}=0$ if $x\nleqslant y$, $Q_{x,x}=1$ and
\begin{center}
$\bm{0}\leqslant \mathsf{deg}  \left(P_{x,y}\right)<\frac{L(y)-L(x)}{2}$.
\end{center}
\end{proposition}

The following is similar to \cite[Section 10]{L03} and \cite[Subsection 3.3]{Y16}. We omit the proof.
\begin{lemma}
Let $x\leqslant y\in E_J$, then
\begin{center}
$q_{x}^{-1}q_{y}\overline{Q_{x,y}}=\sum\limits_{x\leqslant t \leqslant y,t\in E_J}^{}Q_{x,t}\widetilde{R}_{t,y}$.
\end{center}
\end{lemma}

Next, we will show some results which can be proved similar to Section 5 and Section 6. Therefore, we describe only the statement of results and the proofs are omitted.

\begin{definition}
Assume that $x\leqslant y\in E_{J}$ and $\varphi \in \mathscr{M}(x,y)$, we define
\begin{center}
  $\mathscr{R}_\varphi^{*}=\left\{\begin{aligned}
   &q_{x}^{-1}q_{y}\overline {R_{x,y}}                                      &\text{if}\ \ell(\varphi)=1 ,\\
   &U_{(L(x_r)-L(x))/2} \left( q_{x}^{-1}q_{x_r}
    \overline{\mathscr{R}_{\varphi^{'}}^{*}}\right) \mathscr{R}_{x_r,y}^{*}     &\text{if}\ \ell(\varphi)\geqslant2.
   \end{aligned}\right.$
\end{center}
where $\varphi' :x=x_0\leqslant  x_{1}\leqslant \cdots \leqslant x_{r}\in \mathscr{M}(x,x_r)$.
\end{definition}

\begin{theorem} Assume that $x,y\in E_J$.
\begin{enumerate}
  \item If $x < y$, then $Q_{x,y}^{}=L_{\frac{L(y)-L(x)}{2} }\left ( \sum\limits_{\varphi\in \mathscr{J}\left ( x,y \right )}^{}  \widetilde{\mathscr{R}}_\varphi^{*}\right )$.
  \item If $x \leqslant y$, then $Q_{x,y}^{}=\sum\limits_{\varphi\in \mathscr{M}\left ( x,y \right )}^{} \widetilde{\mathscr{R}}_\varphi^{*}$.
\end{enumerate}
\end{theorem}

\begin{theorem}
Assume that $x,y\in E_J$ and $\gamma \in\Gamma''$.
\begin{enumerate}
  \item If $x <y$, then $\left [ q^\gamma  \right ]\left (Q_{x,y}^{}\right )=\left [ q^\gamma \right ]\left (L_{\frac{L(y)-L(x)}{2}}\left ( \sum\limits_{\varphi\in \Psi}^{}  \widetilde{\mathscr{R}}_\varphi^{*}\right )\right )$.
  \item If $x \leqslant y$, then $\left [ q^\gamma \right ]\left (Q_{x,y}^{}\right )=\left [ q^\gamma \right ]\left (\sum\limits_{\varphi\in\Upsilon}^{} \widetilde{\mathscr{R}}_\varphi^{*}\right )$.
\end{enumerate}
\end{theorem}

\begin{theorem}
Assume that $x,y\in E_J$, $\varphi \in \mathscr{M}(x,y)$, $r\in N$ and $\gamma \in\Gamma''$, then
\begin{center}
$\begin{aligned}
\left [ q^\gamma  \right ]\left (Q_{x,y}^{}\right )
=&\left [ q^{\gamma } \right ] \left ( \epsilon _{x}\epsilon _{y}R_{x,y}\right )\\
&+\sum\limits_{r=1}^{\ell(y)-\ell(x)}\sum\limits_{\varphi\in \mathscr{M}\left ( x,y \right )}^{}\sum\limits_{S\in \mathscr{F}_{\gamma }^{*}(\varphi )}^{}\prod\limits_{i=1}^{r}\left [ q_{x}q_{x_{r-i}}^{-1}q^{\lambda_i}q^{\lambda_{i+1}} \right ]\left(\widetilde{R}_{x_{r-i},x_{r+1-i}}\right ),
\end{aligned}$
\end{center}
where we set $S:=(\lambda _0,\lambda _1,\cdots,\lambda _{r+1})$ and
\begin{center}
    $ \begin{aligned}
\mathscr{F}_{\gamma }^{*}(\varphi ) := \{ &(a_0,a_1,\cdots,a_{r+1})\in {(\Gamma'')}^{r+2}\mid \\
              &a_0=L(y)-L(x)-\gamma ,\\
              &\gamma \geqslant a_1> a_2> \cdots>a_r>a_{r+1}=\bm{0}, \\
              &a_i>L(x_{r+1-i})-L(x)-a_i\geqslant a_{i+1}\ for\  i\in\left \{ 1,2,\cdots ,r \right \} \}.
    \end{aligned}$
\end{center}
\end{theorem}


\begin{thebibliography}{99}

\bibitem{BB05}
{\sc A. Bjorner, F. Brenti},
Combinatorics of Coxeter groups, Graduate Texts in Mathematics, 231.
{\it Springer}, New York, 2005.

\bibitem{B94}
{\sc F. Brenti},
A combinatorial formula for Kazhdan-Lusztig polynomials.
{\it Invent. Math.}, {\bf 118} (1994), no. 2, 371--394.

\bibitem{D87}
{\sc V. Deodhar},
On some geometric aspects of Bruhat orderings. II: The parabolic analogue of Kazhdan-Lusztig polynomials.
{\it J. Algebra}, {\bf 111} (1987), no. 2, 483--506.

\bibitem{D91}
{\sc V. Deodhar},
Duality in parabolic set up for questions in Kazhdan-Lusztig theory.
{\it J. Algebra}, {\bf 142} (1991), no. 1, 201--209.

\bibitem{D97}
{\sc V. Deodhar},
J-chains and multichains, duality of Hecke modules, and formulas for parabolic Kazhdan-Lusztig polynomials.
{\it J. Algebra}, {\bf 190} (1997), no. 1, 214--225.

\bibitem{HN12}
{\sc R.B. Howlett, V. Nguyen},
$W$-graph ideals.
{\it J. Algebra}, {\bf 361} (2012), 188--212.

\bibitem{KL79}
{\sc D. Kazhdan, G. Lusztig},
Representations of Coxeter groups and Hecke algebras.
{\it Invent. Math.}, {\bf 53} (1979), no. 2, 165--184.

\bibitem{L83}
{\sc G. Lusztig},
Left cells in Weyl groups. Lecture Notes in Math., 1024,
{\it Springer}, Berlin, 1983.

\bibitem{L03}
{\sc G. Lusztig},
Hecke algebras with unequal parameters. CRM Monograph Series, 18.
{\it American Mathematical Society}, Providence, RI, 2003.

\bibitem{T99}
{\sc H. Tagawa},
A construction of weighted parabolic Kazhdan-Lusztig polynomials.
{\it J. Algebra}, {\bf 216} (1999), no. 2, 566--599.

\bibitem{T01}
{\sc H. Tagawa},
Some properties of inverse weighted parabolic Kazhdan-Lusztig polynomials.
{\it J. Algebra}, {\bf 239} (2001), no. 1, 298--326.

\bibitem{Y15}
{\sc Y. Yin},
$W$-graphs for Hecke algebras with unequal parameters.
{\it Manuscripta Math.}, {\bf 147} (2015), no. 1-2, 43--62.

\bibitem{Y16}
{\sc Y. Yin},
W-graph ideals and duality.
{\it J. Algebra}, {\bf 453} (2016), 377--399.
\end{thebibliography}
\end{document}